\documentclass[twoside,11pt]{amsart} 
\usepackage{epsfig}
\usepackage{latexsym} 
\usepackage{amsmath} 
\usepackage{amssymb}

\newtheorem{theorem}{Theorem}[section]
\newtheorem{lemma}[theorem]{Lemma}

\newdimen\epsfxsize

\newcommand {\gap}     {\makebox[0.075 in]{}}   
   
\newcommand {\fto}     {\longrightarrow}

\newcommand{\cons}[1] {\left<{#1}\right>}
\newcommand{\R}{\mathbb{R}}

\newcommand {\set}[1]  {\left\{ {#1} \right\}}
\newcommand{\Wdg} {\omega}

\newcommand{\PWdg}{\tilde{\omega}}

\begin{document}

\title{Nonsimplicities and the perturbed wedge}
\author{Fred B. Holt}
\address{5520 - 31st Ave NE; Seattle, WA 98105; 
{\rm fbholt62@gmail.com}}

\date{1 Mar 2015}

\begin{abstract}
In producing a counterexample to the Hirsch Conjecture, Francisco Santos has 
described a new construction, the {\em perturbed wedge}.  Santos starts in dimension
$5$ with a counterexample to the nonrevisiting conjecture.  If this counterexample were
a simple polytope, we could repeatedly apply wedging in order to produce a counterexample
to the Hirsch conjecture.  However, the $5$-dimensional polytope is not simple, so we
need a different way to proceed - thus the perturbed wedge.

In the perturbed wedge, we take the wedge over one facet of a nonsimple polytope and then
perturb one or more of the facets of the wedge to move the resulting polytope closer to
simplicity.

In this paper, we examine some of the technical underpinnings of the perturbed wedge, using
Santos and Weibel's counterexample to the Hirsch conjecture as the guiding example.

\end{abstract}


\maketitle


\section*{Introduction}
Coincidence in high dimensions is a delicate issue.  Using nonsimple polytopes
and introducing the
perturbed wedge construction, Santos \cite{Paco} has produced a counterexample to
 the Hirsch conjecture.  As a complement to Santos' exposition, we take the 
 specific counterexample in dimension $20$ provided by Matschke, Santos, and Weibel \cite{Weibel, MSW}, and study some of the
 mechanics behind its construction.
 
 We show first that the $5$-dimensional polytope with which Santos begins his construction
 is a counterexample to the nonrevisiting conjecture \cite{KW, KKd}.
 If Santos and Weibel's counterexamples were simple, 
 then repeated wedging would produce a corresponding counterexample to the 
 Hirsch conjecture \cite{KW}.  Since the counterexample to the nonrevisiting conjecture is
 not simple, we need an alternate method to produce the corresponding counterexample to
 the Hirsch conjecture, and the perturbed wedge provides this method.

Santos presents the perturbed wedge through its dual construction.
A one-point suspension over a vertex in the dual corresponds to a wedge over a facet
in the primal setting, and perturbing a vertex in the dual corresponds to perturbing a facet
in the primal setting.
Here we explore the perturbed wedge as a construction in the primal setting.
This is illustrated in Figure~\ref{SantosWedgeFig}.

Santos' 2010 paper \cite{Paco} makes four important contributions to the 
field of combinatorics and the study of polytopes:
\begin{enumerate}
\item The discovery of a counterexample to the nonrevisiting conjecture for bounded polytopes.  Santos and Weibel construct counterexamples in dimension $5$.  
 
\item Re-introducing nonsimple polytopes into this area of study.  Reduction arguments from the mid-1960s \cite{KW} had enabled researchers in this area \cite{KKd, HKsh, FzH, Hblend}
to focus on simple polytopes.  Santos' counterexample to the nonrevisting conjecture is not a simple polytope, and the perturbed wedge construction requires careful tracking of the nonsimplicities in the course of producing the counterexample to the Hirsch Conjecture.  This work elevates the study of nonsimplicities in polytopes and will enrich the research and pedagogy in this area.

\item Highlighting a gap in the literature, regarding the reduction arguments from nonsimple to simple cases.  We knew previously that for the Hirsch conjecture, if we had a nonsimple counterexample, we could perturb the facets to create a simple counterexample \cite{KW}.  We also knew that for simple polytopes, if we had a counterexample to the nonrevisiting conjecture, we could apply the wedge construction repeatedly to create a counterexample to the Hirsch conjecture.  And we knew that any counterexample to the Hirsch conjecture is also a counterexample to the nonrevisiting conjecture.  However, we did not have a ready prescription on how proceed from a nonsimple counterexample to the nonrevisiting conjecture.

\item Identifying a new construction, the perturbed wedge, that fills this gap.  The perturbed wedge takes a nonsimple counterexample to the nonrevisiting conjecture and produces a counterexample to the Hirsch conjecture. 
\end{enumerate}

\vskip .125in

Using Weibel and Santos' counterexample \cite{Weibel, MSW} to the Hirsch conjecture, we study
the combinatorics behind each application of the perturbed wedge from dimension $5$
through dimension $20$.  We track the effect of the perturbed wedges on the
nonsimplicities.  But first we establish that the $5$-dimensional all-but-simple spindles
of length $6$,
on which Weibel and Santos start their constructions, are counterexamples to the 
nonrevisiting conjecture.

\section{A counterexample to the nonrevisiting conjecture}
A $d$-dimensional spindle $(P,x,y)$ is a polytope with two distinguished 
vertices $x$ and $y$ such that every facet of $P$ is incident to either $x$ or $y$.
The {\em length} of the spindle $(P,x,y)$ is the distance $\delta_P(x,y)$.
The spindle $(P,x,y)$ is {\em all-but-simple} if every vertex of $P$ other than $x$ and $y$ is 
a simple vertex.

A $d$-dimensional all-but-simple spindle $(P,x,y)$ 
of length $d+1$ is a counterexample to the nonrevisiting conjecture.

\begin{lemma} Let $(P,x,y)$ be a $d$-dimensional all-but-simple spindle of length $d+1$.
Then every path from $x$ to $y$ revisits at least one facet.
\end{lemma}

\begin{proof}
Let $X$ be the $n_1$ facets incident to $x$, and let $Y$ be the $n_2$ facets incident to $y$.
Let $\rho = [x,u_1,\ldots,u_{k-1},y]$ be a path from $x$ to $y$ of length $k > d$, with all of the
$u_i$ being simple vertices.

As simple vertices, each $u_i$ is incident to $d$ facets.  
$u_1$ is incident to $d-1$ facets in $X$ and one
facet in $Y$, and $u_{k-1}$ is incident to one facet in $X$ and $d-1$ facets in $Y$.  The incidence
table for $\rho$ looks like the following:
$$
\begin{array}{rccccccc|cccccc}
 & \# {\rm facets} & \multicolumn{6}{c}{X} & \multicolumn{6}{c}{Y}  \\ \hline
x = u_0: & n_1  & 1 & 1 & \cdots & 1 & \cdots  & 1 & & & & & &  \\
u_1: & d   &  &  &  & 1 & \cdots & 1 & 1 & & & & &  \\
  & &  & &  &  & \cdots &  & & \cdots & & & &  \\
u_{k-1}: &  d   & &  &  &  &  & 1 & 1 & \cdots & 1 &  &  &   \\
y = u _k: &  n_2   & &  &  &  &  &  & 1 & \cdots & 1 & \cdots & 1 & 1  \\ \hline
\end{array}
$$

Consider the facet-departures and facet-arrivals from $u_1$ to $u_{k-1}$ 
(the simple part of the path).  
If any of the arrivals were back
to a facet in $X$, this would be a revisit since these facets were all incident 
to the starting vertex $x$.  So the arrivals must all be in $Y$.  

Similarly, all of the departures must be from $X$.  Any departure from a facet in $Y$ would
create a revisit since at $y$ the path would be incident to all of the facets in $Y$.

There are too many arrivals and departures to prevent a revisit.
The vertex $u_1$ is incident to $d-1$ facets in $X$, and since each departure leaves
a facet of $X$, $u_j$ is incident to $d-j$ facets in $X$.  So $u_d$ has completely departed
from $X$, and we must therefore have $u_d=y$.  

Thus if $\rho$ were nonrevisiting, this path would have length at most $d$, 
contradicting the condition that the length $k > d$.
\end{proof}

Santos \cite{Paco} has produced all-but-simple spindles in dimension $5$ of length $6$. 
Currently the smallest example has $25$ facets.
So in dimension $5$ with $25$ facets, we have a counterexample to the nonrevisiting conjecture, with length well below the Hirsch bound. Since the spindle is not simple, at $x$ and $y$, 
the usual way of generating the corresponding counterexample \cite{KW, Hnr} to the Hirsch conjecture (through repeated wedging) does not directly apply, and we need an alternate construction.  

As an aside, we note that Matschke, Santos, and Weibel \cite{MSW} have recently constructed
all-but-simple spindles in dimension $5$ of any length.

\begin{figure}[tb] 
\centering
\includegraphics[width=4.5in]{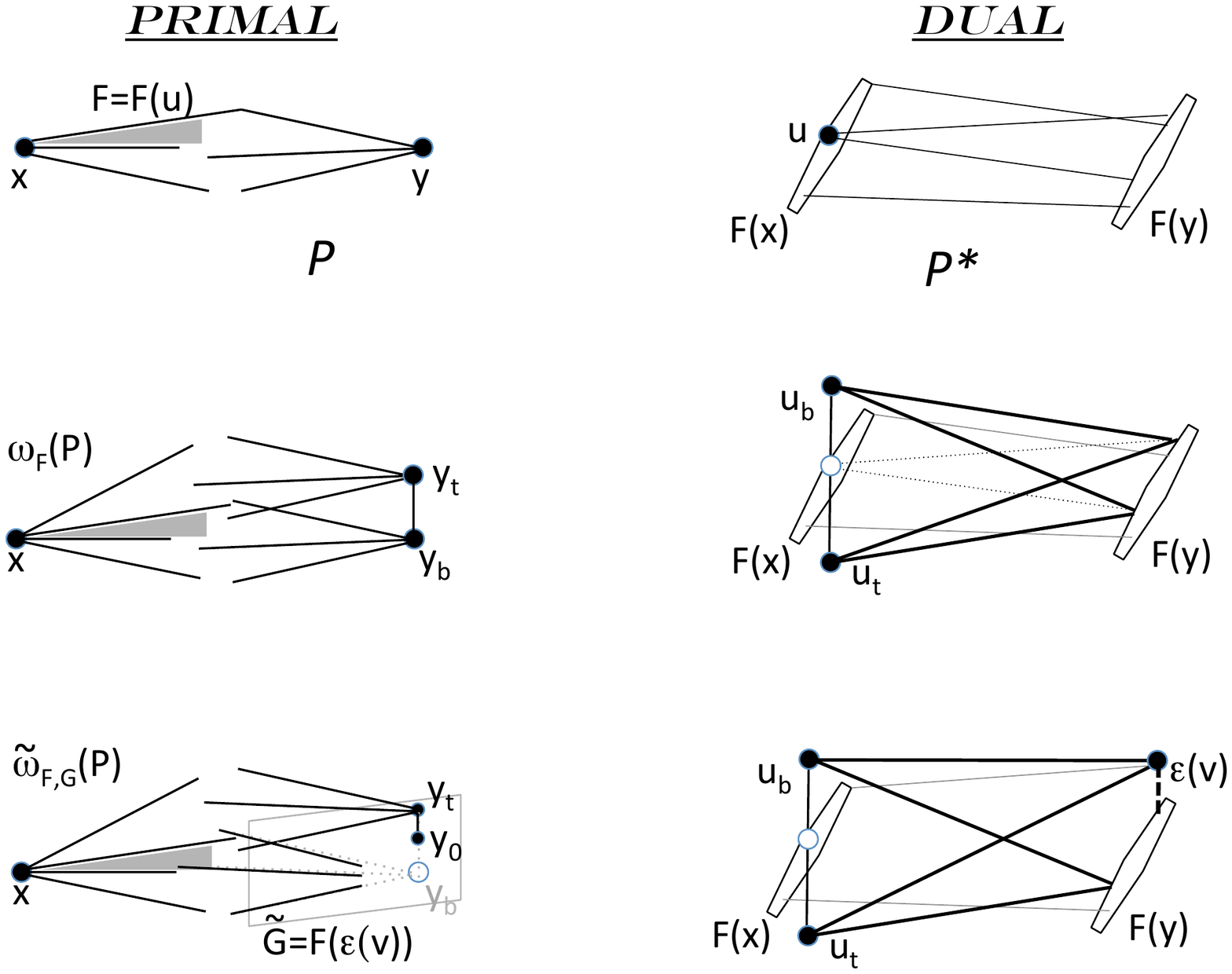}
\caption{\label{SantosWedgeFig} This figure illustrates the perturbed wedge 
construction and its dual construction.  In the primal setting, we first perform a wedge
of $P$ over a facet $F$ which is incident to $x$, followed by a perturbation in the
last coordinate of a facet $G$ incident to $y$.}
\end{figure}

\subsection{A gap in the literature. }
Santos' work has highlighted a gap in the previous literature on the Hirsch conjecture, regarding the equivalence of the simple and nonsimple cases.  We have the previously established results:

\begin{enumerate}
\item If $P$ is a nonsimple counterexample to the Hirsch conjecture of dimension $d$ and $n$ facets, then there is a simple polytope $Q$ also of dimension $d$ and $n$ facets that is also a counterexample to the Hirsch conjecture.  $Q$ is obtained directly from $P$ by perturbing apart its nonsimple vertices.
\item If $P$ is a simple polytope of dimension $d$ and $n$ facets, with two vertices $x$ and $y$ such that every path from $x$ to $y$ includes a revisit to a facets, then there is a simple polytope $Q$ of dimension $n-d$ and $2n-2d$ facets that is a counterexample to the Hirsch conjecture.  $Q$ is obtained directly from $P$ by taking wedges over all (the images of) the $n-2d$ facets of $P$ not incident to either $x$ or $y$.
\item Any counterexample to the Hirsch conjecture is also a counterexample to the nonrevisiting conjecture.
\end{enumerate}

The interesting gap that Santos has addressed is this:  suppose we have a nonsimple counterexample to the nonrevisiting conjecture, then how do we produce the corresponding counterexample to the Hirsch Conjecture?

The $5$-dimensional all-but-simple spindle $(P,x,y)$ produced by Santos is a counterexample
to the nonrevisiting conjecture, as we have verified above, but there are no facets not incident
to either $x$ or $y$.  So we cannot apply the wedges as we would if $P$ were simple.

We also cannot arbitrarily perturb apart the facets incident to $x$ and $y$, to create a simple 
polytope that is a counterexample to the nonrevisiting conjecture.  The revisit on a path may 
well occur along the last edge, the one terminating in $y$.  If we perturb the facets incident to 
$y$, to create simple vertices, some of these revisits may be lost.

\section{The structure of nonsimple faces}

Initially, the spindle is all-but-simple, so the nonsimplicity is concentrated at the two special
vertices $x$ and $y$.  As we iteratively apply wedging, we create nonsimple edges, $2$-faces, and
so on.   There is structure to the nonsimplicity at the images of $x$ and $y$, and we have to
consider how this affects our choice of the facet (or its corresponding vertex in the dual) for 
perturbation.

A $k$-face $F^k$ of a polytope consists of a supporting $k$-dimensional space 
and a boundary.  A facet 
incident to $F^k$ contributes either to defining the supporting space or to defining the boundary.
So relative to a given $k$-face $F^k$, we can partition incident facets into two sets:  
the space-supporting facets and the boundary facets.

For a simple $k$-dimensional face, the supporting space is defined
by the intersection of exactly $d-k$ hyperplanes, and the boundary is given by the
intersection of at least $k+1$ additional hyperplanes with this supporting space. 
The intersections of the boundary hyperplanes with the space-defining hyperplanes and with
each other are all simple intersections.  (We note that vertices are exceptions.  The
$0$-dimensional space is given by the intersection of $d$ hyperplanes, but no additional
facets are required for the boundary.)

{\em Nonsimplicity}  Let $P$ be a $d$-dimensional polytope with facets $H^T$.
A {\em nonsimplicity} of dimension $k$ is a collection of $m$ facets $\tilde{H}^T$ in $H^T$
that satisfy the following conditions:
\begin{enumerate}
\item {\em Nonsimple intersection.} The intersection of the $m$ facets in $\tilde{H}^T$ 
is the supporting space for a $k$-dimensional face $F$ of $P$, with $m > d-k$.
\item {\em Maximality.}  No facet in $H^T \backslash \tilde{H}^T$ contains the supporting space
of $F$.
\item {\em Affine support.}  There is no facet $h^T$ in $\tilde{H}^T$ such that
the intersection of the facets $\tilde{H}^T \backslash h^T$ is more than $k$-dimensional.
\end{enumerate}

We call this last condition the requirement for {\em affine support} because in the dual setting,
we may know that the points $\tilde{H}$ are affinely dependent, but there may be points
in $\tilde{H}$ that do not contribute to this affine dependence.  That is, for any affine dependence
$\tilde{H} \cdot \alpha = 0$, the coordinate in $\alpha$ corresponding to $h$ is always $0$.
We will see examples of this below.

The {\it excess} of a nonsimplicity is the quantity $m-(d-k)$, the number of facets supporting
the face beyond those necessary in a simple polytope.

{\em Nonsimplicities in a nonsimple polytope.} 
A nonsimple polytope may contain several nonsimplicities, and these nonsimplicities
may be of different dimensions.  

Let's consider a few examples.  A pyramid over a 
hexagon is a $3$-dimensional polytope, whose apex $v$ is a nonsimple vertex,
but the spaces supporting all of the edges and $2$-faces are simple.  The prism over
this pyramid is a $4$-dimensional polytope whose only nonsimplicity is the edge
$[v_b, v_t]$.  The vertices $v_b$ and $v_t$ are not simple vertices, but in each case
the nonsimple vertex is the intersection of one facet with the nonsimplicity in the vertical edge.
So although the vertex $v_t$ is not simple, the facets incident to $v_t$ do not satisfy the
condition of affine support; in particular the top facet does not contribute to the nonsimple edge.

For the next example, consider
our all-but-simple spindle $(P,x,y)$.  When we perform the wedge over a facet $F$ incident
to $x$, the nonsimple vertex $y$ generates a nonsimplicity in the vertical edge $[y_b, y_t]$.  
This edge contains the nonsimplicity, and its two vertices consist of the intersection of the
nonsimplicity supporting the vertical edge with the top facet $T$ and with the bottom facet $B$.  

Suppose we wanted to reduce the nonsimplicity of this edge $[y_b,y_t]$ by perturbing a facet.
If we perturb $T$ or $B$, the nonsimplicity is not changed.  We must perturb a facet that is
contributing to the nonsimplicity.  That is, to reduce the nonsimplicity of the edge $[y_b, y_t]$
we must perturb one of its space-supporting facets.

Continuing this example, if we now construct the wedge over a facet incident to the image of $x$ in the first wedge, the nonsimple edge $[y_b,y_t]$ generates a nonsimple $2$-face.  If the foot of this wedge is $T$ or $B$, then the image of $[y_b, y_t]$ is a triangle; otherwise it is a rectangle.
In either case the nonsimplicity occurs in the supporting $2$-space.  

The nonsimplicity of a nonsimple $k$-face can arise in any dimension from $0$ to $k$.
A nonsimple $k$-face could have a combination of various nonsimplicities across these dimensions.  

This means that when we choose a facet to perturb as 
part of the construction of the perturbed wedge, it won't suffice to choose any facet
incident to the $k$-face.  We need to choose a facet that is contributing to some nonsimplicity.  

This approach to nonsimplicity borrows heavily from the insights underlying the Gale transform
\cite{Gr, Zg}.

\renewcommand{\arraystretch}{1.4}

\section{Applying the perturbed wedge to all-but-simple spindles}

We will study the perturbed wedge by following the proof of the strong $d$-step
theorem for prismatoids (Theorem 2.6 in \cite{Paco}) and tracking the nonsimplicities
through the specific construction
of the counterexample to the Hirsch conjecture \cite{Weibel, MSW}. 

In Santos and Weibel's specific construction, they reuse the base of each wedge as the foot for 
the next wedge.  This simplifies the exposition.  

Let $(P_5,x,y)$ be the $5$-dimensional all-but-simple spindle with $25$ facets and
length $6$.
We can write the matrix of inward-pointing
normals $H^T(P_5)$ for the facets as
$$ H^T(P_5) = \left[ \begin{array}{c} H^T(\tilde{X}) \\ h_{12}^T \\ \hline
H^T(\tilde{Y}) \\ h_{33}^T
\end{array} \right].$$
For the vertices $x$ and $y$ we have
$$\begin{array}{lcl}
\left[ \begin{array}{c} H^T(\tilde{X}) \\ h_{12}^T \\ \hline
H^T(\tilde{Y}) \\ h_{33}^T
\end{array} \right] \cdot x = \left[\begin{array}{c} \cons{0} \\ 0 \\ \hline
\cons{+} \\ + \end{array}\right] & {\rm and} &
\left[ \begin{array}{c} H^T(\tilde{X}) \\ h_{12}^T \\ \hline
H^T(\tilde{Y}) \\ h_{33}^T
\end{array} \right] \cdot y = \left[\begin{array}{c} \cons{+} \\ + \\ \hline
\cons{0} \\ 0 \end{array}\right] \end{array}
$$ 
We assume that that we are embedding the polytope in the flat $1\times \R^5$
as the solution to $H^T v \ge \cons{0}$.

For $(P_5,x,y)$, the vertex $x$ is incident to the twelve facets $h^T_1,\ldots,h^T_{12}$.
We reserve the indices $h_{13}^T,\ldots,h_{20}^T$ for the wedges at $x$ to
stay consistent with Santos and Weibel's counterexample.  The vertex
$y$ is incident to the thirteen facets $h^T_{21}, \ldots,h^T_{33}$.  The wedges
at $y$ will fill the indices from $h_{34}^T, \ldots, h_{40}^T$.
Among the initial hyperplanes at $x$ and $y$, we isolate the
facets $h_{12}^T$ and $h_{33}^T$ since Santos and Weibel apply all the wedges at $x$
using images of $h_{12}^T$ as the feet, and all the wedges at $y$ using 
images of $h_{33}^T$ as the feet. 

Each application of the perturbed wedge construction consists of two operations.
First we take the wedge over a facet incident to the image of $x$ (or $y$ respectively),
which adds one dimension and one facet to the polytope.
Then we perturb one or more facets incident to $y$ (or $x$ respectively), in the new dimension.

Since the action of each perturbed wedge on $H^T$ occurs in a separate new dimension, 
we can reorder the columns of $H^T$ as Santos and Weibel did, to apply the eight wedges
at $x$ first (and the corresponding perturbations to facets incident to $y$), 
and then apply the seven wedges at $y$ (and the perturbations to facets at $x$).

\subsection{First perturbed wedge: $d=6$.}
For the first application of the perturbed wedge to $P_5$, Santos and Weibel take the wedge over
$h_{12}^T$ and perturbs $h_{32}^T$.

After the wedge over the facet $h_{12}^T$ we have
\begin{eqnarray*}
H^T(\Wdg P_5) & = &
 \left[ \begin{array}{cc} H^T(\tilde{X}) & \cons{0} \\
  h^T_{12} & C \\
  h^T_{12} & -C \\ \hline
  H^T(\tilde{Y}) & \cons{0} \\
  h^T_{33} & 0 \end{array} \right]
\end{eqnarray*}
Santos and Weibel use $C= 10^7$.  Since these are inward-pointing normals, the new facet
$h_{12}^T = \left[ h_{12}^T \gap C \right]$ is the base of the wedge, and the new facet
$h_{13}^T = \left[ h_{12}^T \gap -C \right]$ is the top of the wedge.

Vertices of $P_5$ that are incident to the foot $h_{12}^T$ of this wedge each have
one natural image.  In coordinates we can write this
$$ \left[ v \right] \fto \left[ \begin{array}{c} v \\ 0 \end{array} \right].$$
Vertices of $P_5$ that are not incident to the foot each have two natural images,
one in the base, and one in the top, connected by a vertical edge.
$$ \left[ v \right] \fto v_b = \left[ \begin{array}{c} v \\  -h_{12}^Tv / C \end{array} \right],
\gap {\rm and} \gap v_t = \left[ \begin{array}{c} v \\  h_{12}^Tv / C \end{array} \right].$$

The facets other than the top and base all have natural images as vertical facets in the wedge.

After the wedge, the natural images of the nonsimple vertices $x$ and $y$
are the nonsimplicities $x$ and the vertical edge $[y_b,y_t]$.  This vertical
edge is an example of a nonsimplicity with a simple boundary -- both vertices are the
intersection of the $1$-dimensional space supporting the edge with a single additional
facet.  Although both $y_b$ and $y_t$ are nonsimple vertices, the nonsimplicity is the vertical
edge between them.  The vertex $y_t$ is the intersection between the nonsimple edge 
and the top,
and $y_b$ is the intersection between the nonsimple edge and the base.  

\begin{table}
\begin{tabular}{cccccc} \hline
$P$ & ${\rm dim}P$ & $S$ & ${\rm dim}S$ & \# supporting facets & excess \\
 & $d$ & & $k$ & $m$ & $m-(d-k)$ \\ \hline
 $P_5$ & $5$ & $x$ & $0$ & $12$ & $7$ \\
  & & $y$ & $0$ & $13$ & $8$ \\ \hline
  $\Wdg P_5$ & $6$ & $x$ & $0$ & $13$ & $7$ \\
   & & $[y_b, y_t]$ & $1$ & $13$ & $8$ \\ \hline
  $\PWdg P_5$ & $6$ & $x$ & $0$ & $13$ & $7$ \\
   & & $[y_0, y_t]$ & $1$ & $12$ & $7$ \\ \hline
\end{tabular}
\caption{\label{P5Table}
Data on nonsimplicities through both steps of the first perturbed wedge.
Although both $y_b$ and $y_t$ are nonsimple vertices, the nonsimplicity is the vertical
edge between them.  $y_t$ is the intersection between the nonsimple edge and the top,
and $y_b$ is the intersection between the nonsimple edge and the base.  The perturbation
of one of the facets supporting the nonsimple edge reduces the excess of this edge, and
the perturbed facet truncates away $y_b$, terminating the edge at $y_0$.}
\end{table}

We now perturb $h_{32}^T$, which is one of the facet supporting the nonsimple
edge $[y_b,y_t]$.  Since the perturbation is positive (Santos and Weibel use $\epsilon=1$),
the inward-pointing normal for $h_{32}^T$ is tipped up.  At the edge $[y_b,y_t]$
the perturbed facet truncates away $y_b$, terminating the nonsimple vertical edge
at $y_0$.
$$ y_0 = \left[\begin{array}{c} y \\ 0 \end{array}\right].$$
The perturbed facet no longer supports the nonsimple edge, and it replaces the
base as a boundary facet for this edge.

The effect of this first perturbed wedge, taking the wedge over a facet at $x$ and
perturbing a facet at $y$, is to
reduce the excess of the nonsimplicity at $y$ by $1$.  See Table~\ref{P5Table}.

For this perturbed wedge, we can write the matrix of inward-pointing normals as
\begin{eqnarray*}
H^T(\PWdg P_5) & = &
 \left[ \begin{array}{cc} H^T(\tilde{X}) & \cons{0} \\
  h^T_{12} & C \\
  h^T_{12} & -C \\ \hline
  H^T(\tilde{Y}) & \epsilon_{32} \\
  h^T_{33} & 0 \end{array} \right]
\end{eqnarray*}
By $\epsilon_{32}$ we mean the vector all of whose entries are $0$ except for
the entry indexed by $32$, whose entry is $\epsilon$.  We again note that Santos and Weibel
use $C=10^7$ and $\epsilon=1$.

Since the nonsimplicity in $\Wdg P_5$ is an edge of excess $8$, perturbing a single
facet, e.g. $h_{32}^T$, shifts this one hyperplane from supporting the space of the edge
to forming part of the boundary of the edge.  But the nonsimplicity in $\PWdg P_5$
is still the $1$-dimensional coincidence of twelve hyperplanes.

\subsection{Second perturbed wedge: $d=7$}
For the second application of the perturbed wedge, Santos and Weibel use the base $h_{12}^T$
from the previous wedge as the foot, and we perturb $h_{31}^T$.  The construction
is much like the previous one, and we have:
\begin{eqnarray*}
H^T(\PWdg^2 P_5) & = &
 \left[ \begin{array}{ccc} H^T(\tilde{X}) & \cons{0} & \cons{0} \\
  h^T_{12} & C & C \\
  h^T_{12} & -C & 0 \\
  h^T_{12} & C & -C \\ \hline
   H^T(\tilde{Y}) & \epsilon_{32} & \epsilon_{31} \\
  h^T_{33} & 0 & 0 \end{array} \right]
\end{eqnarray*}
The new $h_{12}^T$ is again the base of the wedge, and 
$h_{14}^T = [h_{12}^T \gap C \gap -C]$ is the top.

For the nonsimplicities we still have the nonsimple vertex $x$ with excess $7$.
Down at $Y$ the nonsimplicity is now a rectangle.  Under the wedge, the nonsimple
vertical edge $[y_0,y_t]$ has for its natural image a rectangle 
${\rm Hull}(y_{0b}, y_{0t}, y_{tb}, y_{tt})$.  That is, $y_0$ has one natural image $y_{0b}$
in the base and another $y_{0t}$ in the top; similarly $y_t$ has natural images $y_{tb}$ and 
$y_{tt}$ in the base and top respectively.  Perturbing $h_{31}^T$ in the new dimension
truncates away $y_{0b}$ and $y_{tb}$.  

The nonsimplicity at $Y$ is now the rectangle
${\rm Hull}(y_{00}, y_{0t},y_{t0},y_{tt})$.  The two-dimensional supporting space is
supported by eleven facets $h_{21}^T,\ldots,h_{30}^T$ and $h_{33}^T$, and it has a
simple boundary consisting of the four facets $h_{13}^T, \; h_{14}^T,$ and
$h_{31}^T, \; h_{32}^T$.

\begin{table}
\begin{tabular}{cccccc} \hline
$P$ & ${\rm dim}P$ & $S$ & ${\rm dim}S$ & \# supporting facets & excess \\
 & $d$ & & $k$ & $m$ & $m-(d-k)$ \\ \hline
   $\PWdg P_5$ & $6$ & $x$ & $0$ & $13$ & $7$ \\
   & & $[y_0, y_t]$ & $1$ & $12$ & $7$ \\ \hline
   $\PWdg^2 P_5$ & $7$ & $x$ & $0$ & $14$ & $7$ \\
   & & ${\rm Hull}(y_{00}, y_{t0},y_{tt},y_{0t})$ & $2$ & $11$ & $6$ \\ \hline
\end{tabular}
\caption{\label{P7Table}
After the second perturbed wedge, the nonsimplicity in $Y$ is now of dimension $2$.
By perturbing one of the supporting hyperplanes, we reduce the excess of this nonsimplicity
by one.}
\end{table}

\subsubsection{We cannot choose any facet at $X$ to be the foot of the wedge.}
In iterating the perturbed wedge, we discover that there are restrictions on our choice
of the foot for the next wedge.  There are combinatorial circumstances under which the 
distance from $y_0$ to $x$
may not be increased under the perturbed wedge.

\begin{lemma}\label{BadLemma}  
Let $y$ be a nonsimple vertex of a $d$-dimensional polytope $P$.  Let $F$
be a facet of $P$ not incident to $y$, and let $G$ be a facet incident to $y$.
We perform the perturbed wedge, wedging over $F$ and perturbing the image of $G$.

If there was a nonsimple edge in $G$ from $y$ to a vertex $w$ in $F \cap G$, then 
 after the perturbed wedge this edge has a natural image $[w,y_0]$.
\end{lemma}

\begin{proof}
As a nonsimple edge, the 1-dimensional space of $[w,y]$ is defined by 
the coincidence of $G$ and at least $d-1$ other facets $\hat{Y}$.  
See Figure~\ref{NbrVtxFig}. 
These facets $\hat{Y}G$
are incident to both $w$ and $y$.  The boundary of the edge at $w$ is established by
$F$ and possibly more facets $\hat{X}$, none of which can be incident to $y$.  The boundary
of the edge at $y$ is established by the facets $Y\setminus \hat{Y}$.

Under the wedge $\Wdg_F P$, the image of the edge is a nonsimple $2$-dimensional face,
the triangle with vertices $w$, $y_t$, and $y_b$.

Now, under the perturbation, the space of the triangular face is still defined by $\hat{Y}$.
$\tilde{G}$ intersects this face in the plane with last coordinate $0$, creating an edge from
$w$ to $y_0$.
\end{proof}

\begin{figure}[tb] 
\centering
\includegraphics[width=4.5in]{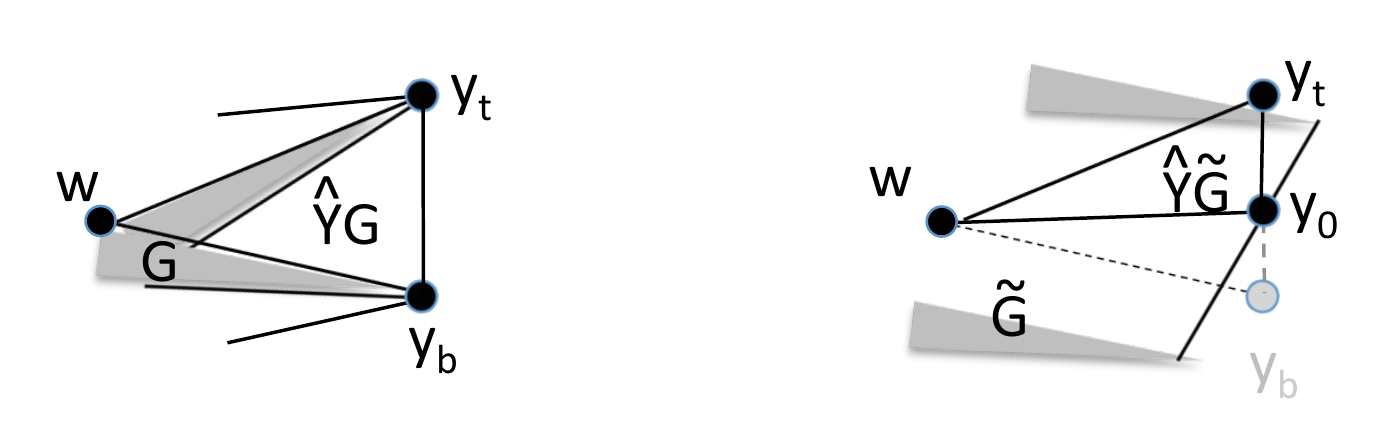}
\caption{\label{NbrVtxFig} Suppose that in the perturbed
wedge we perturb a facet $G$ supporting this nonsimple edge $[w,y]$.
Let the facets supporting the edge be $\hat{Y}G$ and those supporting $y$
be $YG$ with $\hat{Y}\subset Y$.
Under the wedge, the vertex $y=YG$ becomes an edge and
$\hat{Y}G$ becomes a $2$-face.  The vertex $y$ has two natural images, $y_t$ which is 
incident to the facets $YGT$, and $y_b$ which is incident to $YGB$.  The facet $F$ is 
replaced by two facets, the top $T$ and the base $B$.
Now we perturb the facet $G$, introducing a small positive value 
in the last coordinate of its inward-pointing normal.
Since $y$ was nonsimple, the vertex $y_t$ remains, but the vertex $y_b$ is truncated away.
Instead, $\tilde{G}$ now intersects the vertical edge $Y$ in the plane 
$\set{1}\times \R^d \times \set{0}$, at a new vertex $y_0$.  
The $2$-face $\hat{Y}G$  was nonsimple, and so the
$2$-face remains with its space supported by $\hat{Y}$, and $\tilde{G}$ intersects
it in an edge $[w,y_0]$.}
\end{figure}

If the edge $[w,y]$ was fast toward $x$ in $P$, then the edge $[w,y_0]$ would be fast
toward $x$ in $\PWdg P$, and the distance would not be increased:
$\delta_P(x,y)=\delta_{\PWdg P}(x,y_0)$.

As a specific example, we took the second wedge over the base $h_{12}^T$ 
of the previous wedge. This choice works since in perturbing $h_{32}^T$ we had
truncated away $y_b$.  We could not have chosen the top $h_{13}^T$ of 
the previous wedge as the foot of the second wedge, since $[y_t,y_0]$ satisfies the
conditions of the lemma, with $y_t=w$ and $y_0=y$.  If we had used $h_{13}^T$
as the foot of the second wedge, the distance from $x$ to $y_0$ would not have
increased under the perturbed wedge construction.

Since the initial polytope $P_5$ is an all-but-simple spindle, the conditions of this lemma
do not arise on the first application of the perturbed wedge.  However, in iterating the 
perturbed wedge, the dimension of the nonsimplicity grows, and we must choose
the feet of the wedges so as to avoid edges that satisfy the conditions of 
Lemma~\ref{BadLemma}.

\subsubsection{We cannot choose any facet at $Y$ to perturb.}
In order to reduce the excess of the nonsimplicity at $Y$, we have to perturb
a facet that is contributing to the nonsimplicity, that is, a facet that supports the
space of the nonsimplicity. 

In the original polytope $P_5$, the nonsimplicity in $Y$ is a single nonsimple vertex $y$,
so we can choose any hyperplane.  However, even in $\PWdg P_5$ the vertex $y_0$ is the
intersection of a nonsimple edge with the previously perturbed facet $h_{32}^T$.  So
$h_{32}^T$ does not support the nonsimplicity and further perturbing $h_{32}^T$ will not reduce
the excess.

In forming $\PWdg^2 P_5$, after the wedge the $2$-dimensional nonsimplicity is supported
by twelve facets $h_{21}^T,\ldots,h_{31}^T$, and $h_{33}^T$; and it has four
boundary facets $h_{12}^T, h_{13}^T, h_{14}^T$ and $h_{32}^T$.  So to complete
the construction of $\PWdg^2 P_5$, we have to perturb one of the twelve facets
supporting the $2$-dimensional nonsimplicity.  This perturbed facet replaces 
the base $h_{12}^T$ as part of the boundary of the nonsimplicity.  This leaves eleven
facets supporting the $2$-dimensional nonsimplicity, for an
excess of $11-(7-2)=6$.

\subsubsection{If we proceed in this way, the nonsimplicity at $Y$ becomes a $k$-dimensional
hypercube of excess $8-k$.}
If we continue applying the perturbed wedge by 
wedging over the base $h_{12}^T$ of the previous 
wedge and then perturbing a single facet in $Y$, then the
nonsimplicity at $Y$ grows as a $k$-dimensional hypercube
of excess $8-k$ and simple boundary.
The vertices of this hypercube $I^k$ are the natural images of $y$ indexed
by $0$'s and $t$'s from $y_{0\cdots 0}$ to $y_{t \cdots t}$.

Actually the facet $h_{12}^T$ plays no special role.  If we use any facet not incident
to any of the vertices in $I^k$, the nonsimplicity grows as a hypercube $I^k$.  Any facet
incident to the hypercube would satisfy the conditions of Lemma~\ref{BadLemma},
so we can not use these.

However, Santos and Weibel do something interesting in the next perturbed wedge.  They perturb
two facets, which provides something new for our studies.

\subsection{Third perturbed wedge:  perturbing two facets in $d=8$}
For the third application of the perturbed wedge, Santos and
Weibel again use $h_{12}^T$ as the base, but we now perturb
{\em two} hyperplanes, $h_{21}^T$ and $h_{30}^T$.

\begin{eqnarray*}
H^T(\PWdg^3 P_5) & = &
 \left[ \begin{array}{rcccc} & H^T(\tilde{X}) & \cons{0} & \cons{0} & \cons{0} \\
 &  h^T_{12} & C & C & C \\
{\scriptstyle h^T_{13}=}  & h^T_{12} & -C & 0 & 0 \\
{\scriptstyle h^T_{14}=}  & h^T_{12} & C & -C & 0 \\   
{\scriptstyle h^T_{15}=}  & h^T_{12} & C & C & -C \\ \hline
 &  H^T(\tilde{Y}) & \epsilon_{32} & \epsilon_{31} 
    & \epsilon_{21}+\epsilon_{30} \\
  & h^T_{33} & 0 & 0 & 0 \end{array} \right]
\end{eqnarray*}

The new $h_{12}^T$ is again the base of the wedge, and 
$h_{15}^T = [h_{12}^T \gap C \gap C \gap -C]$ is the top.
The nonsimplicity $I^2$ becomes a $3$-dimensional hypercube
with vertices
$$y_{00b}, y_{00t}, y_{0tb}, y_{0tt}, y_{t0b}, y_{t0t}, y_{ttb}, y_{ttt}.$$

Since we have made positive perturbations to $h_{21}^T$ and $h_{30}^T$,
both of these truncate away the nonsimple vertices in the base, replacing them
with vertices with last coordinate $0$:
$$y_{000}, y_{00t}, y_{0t0}, y_{0tt}, y_{t00}, y_{t0t}, y_{tt0}, y_{ttt}.$$

The effect of perturbing two hyperplanes is to partition the nonsimplicity in $Y$.
Of the eleven hyperplanes supporting the $3$-dimensional space, we have shifted
two of them to the boundary.  One would have sufficed to create this section of the 
boundary, so the second hyperplane is excessive.

\begin{table}
\begin{tabular}{cccccc} \hline
$P$ & ${\rm dim}P$ & $S$ & ${\rm dim}S$ & \# supporting facets & excess \\
 & $d$ & & $k$ & $m$ & $m-(d-k)$ \\ \hline
   $\PWdg^2 P_5$ & $7$ & $x$ & $0$ & $14$ & $7$ \\
   & & $I^2$ & $2$ & $11$ & $6$ \\ \hline
   $\PWdg^3 P_5$ & $8$ & $x$ & $0$ & $15$ & $7$ \\
   & & $I^3$ & $3$ & $9$ & $4$ \\ 
   & & $I^2$ & $2$ & $11$ & $5$ \\ \hline
\end{tabular}
\caption{\label{P8Table}
After the third perturbed wedge, since we perturbed two hyperplanes,
the nonsimplicity in $Y$ is partitioned into two nonsimplicities $I^2\subset I^3$.
The excess of the $3$-dimensional nonsimplicity drops by two, but one of these
is picked back up in the $2$-dimensional nonsimplicity. }
\end{table}

As we proceed, we will need to perturb facets out of both of these nonsimplicities.
Santos and Weibel take a specific strategy.  In the fourth perturbed wedge, we perturb
$h_{21}^T$ and $h_{22}^T$.  We continue perturbing overlapping pairs,
until we exhaust the excess of the nonsimplicities in $Y$.

\subsection{Finishing the perturbed wedges with feet at $X$.}
We need to apply a total of $8$ perturbed wedges with feet in $X$ in oder 
to reduce the excess of the nonsimplicity in $Y$ to $0$.
When this set of perturbed wedges are completed, the inward-pointing
normals form the $33\times 14$ matrix $H^T(\PWdg^8 P_5)$, shown in
Figure~\ref{HP13Fig}.

\begin{figure} \[
\left[ \begin{array}{rccccccccc}
 & H^T(\tilde{X}) & \cons{0} & \cons{0} & \cons{0} & \cons{0} & \cons{0} & \cons{0} & \cons{0} & \cons{0} \\
 &  h^T_{12} & C & C & C & C & C & C & C & C  \\
{\scriptstyle 13:}  & h^T_{12} & -C & 0 & 0 & 0 & 0 & 0 & 0 & 0 \\
{\scriptstyle 14:}  & h^T_{12} & C & -C & 0 & 0 & 0 & 0 & 0 & 0 \\   
{\scriptstyle 15:}  & h^T_{12} & C & C & -C  & 0 & 0 & 0 & 0 & 0\\ 
{\scriptstyle 16:}  & h^T_{12} & C & C &  C & -C & 0 & 0 & 0 & 0\\ 
{\scriptstyle 17:}  & h^T_{12} & C & C & C & C & -C & 0 & 0 & 0\\ 
{\scriptstyle 18:}  & h^T_{12} & C & C & C & C & C & -C & 0 & 0\\ 
{\scriptstyle 19:}  & h^T_{12} & C & C & C & C & C & C & -C  & 0 \\ 
{\scriptstyle 20:}  & h^T_{12} & C & C & C & C & C & C & C & -C \\  \hline
 &  H^T(\tilde{Y}) & \epsilon_{32} & \epsilon_{31} &
    {\scriptstyle \epsilon_{21}+\epsilon_{30}} & {\scriptstyle \epsilon_{21}+\epsilon_{22}} & 
     {\scriptstyle \epsilon_{22}+\epsilon_{23}} & {\scriptstyle \epsilon_{23}+\epsilon_{24}} &
     {\scriptstyle \epsilon_{24}+\epsilon_{25}} & {\scriptstyle \epsilon_{25}+\epsilon_{26}} \\
  & h^T_{33} & 0 & 0 & 0 & 0 & 0 & 0 & 0 & 0 \end{array} \right] \]
\caption{\label{HP13Fig}The $33 \times 14$ matrix $H^T(\PWdg^8 P_5)$
of inward-pointing normals, after we have completed the perturbed wedges with
feet in $X$.  The nonsimplicity in $Y$ has been eliminated.}
\end{figure}

Due to the regularity of the applications of the perturbed wedge -- wedging over $h_{12}^T$
and perturbing over a pair of facets, one of which was in a previous pair -- we make just
two observations about the rest of these perturbed wedges.

\subsubsection{Transfering the nonsimplicity.}
Let's consider the first overlapping pair of perturbations.  In forming $\PWdg^3 P_5$
we perturbed $h_{21}^T$ and $h_{30}^T$.  This partitioned the nonsimplicity in $Y$,
making one of its boundary faces itself a nonsimplicity.  

Now in forming
$\PWdg^4 P_5$ we perturb $h_{21}^T$ and $h_{22}^T$.  
The larger nonsimplicity is now $4$-dimensional.  The second perturbation on
$h_{21}^T$ moves this hyperplane into a diagonal position, so that it supports
a $2$-dimensional face on the boundary of $I^4$.  So the nonsimplicities in $Y$
are now partitioned into a $4$-dimensional nonsimplicity $I^4$ of excess $3$, one of whose
$2$-dimensional faces $I^2$ has one extra supporting hyperplane, for excess $4$.

We list the supporting hyperplanes here, to illustrate the transitions of the nonsimplicities
through the next two applications of the perturbed wedge.
\[
\begin{array}{cccl}
P & S  & {\rm excess} &  \; h_j^T \; {\rm supporting} \; S \\ \hline
\PWdg^3 P_5 & I^3 & 4 & h_{22}^T, \ldots, h_{29}^T, h_{33}^T \\
 & I^2 & 5 & h_{21}^T,\ldots,h_{30}^T, h_{33}^T = I^3 \cap h_{21}^T \cap h_{30}^T \\ \hline
\PWdg^4 P_5 & I^4 & 3 & h_{23}^T, \ldots, h_{29}^T, h_{33}^T \\
  & I^2 & 4 & h_{21}^T,\ldots,h_{30}^T, h_{33}^T 
  = I^4 \cap h_{21}^T \cap h_{22}^T \cap h_{30}^T \\ \hline 
\PWdg^5 P_5 & I^5 & 2 & h_{24}^T, \ldots, h_{29}^T, h_{33}^T \\
  & I^2 & 3 & h_{21}^T,\ldots,h_{30}^T, h_{33}^T 
  = I^5 \cap h_{21}^T \cap h_{22}^T \cap h_{23}^T \cap h_{30}^T \\  \hline 
\end{array}
\]

Santos and Weibel's construction provides a nice opportunity to study some variation in the
way the nonsimplicities evolve under this strategy for the perturbations in $Y$.
The larger nonsimplicity vanishes -- that is, the excess is eliminated -- after seven
applications of the perturbed wedge.  Refer to Table~\ref{PkTable}.  The 
remaining nonsimplicity is $2$-dimensional.
It's interesting to note that the last perturbation doesn't require two facets 
to be perturbed, since the excess is only $1$; but perturbing the pair
$h_{25}^T, \; h_{26}^T$ completes a nice banded aspect to the construction.

\begin{table}
\begin{tabular}{cccccc} \hline
$P$ & ${\rm dim}P$ & $S$ & ${\rm dim}S$ & \# supporting facets & excess \\
 & $d$ & & $k$ & $m$ & $m-(d-k)$ \\ \hline
    $\PWdg^3 P_5$ & $8$ & $x$ & $0$ & $15$ & $7$ \\
   & & $I^3$ & $3$ & $9$ & $4$ \\ 
   & & $I^2$ & $2$ & $11$ & $5$ \\ \hline
   $\PWdg^4 P_5$ & $9$ & $x$ & $0$ & $16$ & $7$ \\
   & & $I^4$ & $4$ & $8$ & $3$ \\ 
   & & $I^2$ & $2$ & $11$ & $4$ \\ \hline
   $\PWdg^5 P_5$ & $10$ & $x$ & $0$ & $17$ & $7$ \\
   & & $I^5$ & $5$ & $7$ & $2$ \\ 
   & & $I^2$ & $2$ & $11$ & $3$ \\  \hline
   $\PWdg^6 P_5$ & $11$ & $x$ & $0$ & $18$ & $7$ \\
   & & $I^6$ & $6$ & $6$ & $1$ \\ 
   & & $I^2$ & $2$ & $11$ & $2$ \\ \hline
   $\PWdg^7 P_5$ & $12$ & $x$ & $0$ & $19$ & $7$ \\
   & & $(I^7)$ & $7$ & $5$ & ${\bf 0}$ \\ 
   & & $I^2$ & $2$ & $11$ & $1$ \\ \hline
   $\PWdg^8 P_5$ & $13$ & $x$ & $0$ & $20$ & $7$ \\
   & & $(I^2)$ & $2$ & $11$ & ${\bf 0}$ \\ \hline
\end{tabular}
\caption{\label{PkTable}
We tabulate the nonsimplicities through the eight applications of the perturbed wedge
with $h_{12}^T$ as foot and perturbations in $Y$.}
\end{table}

\subsection{Perturbed wedges with feet in $Y$.}
Now that the nonsimplicity in $Y$ has been eliminated, we have to apply perturbed wedges
with feet in $Y$ in order to eliminate the nonsimplicity in $X$.  The nonsimplicity in $X$ is
still concentrated at the single nonsimple vertex $x$, which has excess $7$.

The approach is almost identical.  We apply the wedges with $h_{33}^T$ as foot, indexing the
base as the new $h_{33}^T$ and the top as a new facet.  We apply perturbations in $X$
with the first three being of single facets, and the next four being of overlapping pairs of facets.

There is one major exception in this presentation of Santos and Weibel's construction.  We have
laid out the order of operations as indicated by the columns of $H^T$, completing the
wedges with feet in $X$ and then turning to the wedges with feet in $Y$.
The first perturbation in $X$ arises more naturally if we performed this operation first, 
perturbing $h_{12}^T$ before using it as the foot of eight wedges.
If we had taken the very first wedge over $h_{33}^T$, 
the effect of the first operation would ripple
automatically through the natural images of $h_{12}^T$.  

Instead, in our first perturbed
wedge with foot in $Y$, we perturb {\em all} of the facets $h_{12}^T,\ldots,h_{20}^T$.
We call this column out when we write down the full $40\times 21$ matrix of 
inward-pointing normals.
\[
H^T(\PWdg^7 \PWdg^8 P_5) = 
\left[ \begin{array}{ccccc}
 H^T(\tilde{X}) & \cons{0} & \cons{0} & \epsilon_X \\
 \cons{h_{12}^T} & W_8 & \cons{\epsilon} & \cons{0} \\ \hline
 H^T(\tilde{Y}) & \epsilon_Y & \cons{0} & \cons{0} \\
 \cons{h_{33}^T} & \cons{0} & \multicolumn{2}{c}{W_7} 
\end{array} \right]
\]
Here $W_8$ is the block that encodes the eight wedges over $h_{12}^T$,
and $W_7$ is a block of the same structure encoding the seven wedges
over $h_{33}^T$.  The block $\cons{\epsilon}$ is a $9\times 1$ column
of a constant $\epsilon$.
For the blocks $\epsilon_Y$ and $\epsilon_X$ we have
$$ \epsilon_Y = 
 \left[ \begin{array}{cccccccc}
\epsilon_{32} & \epsilon_{31} &
     \epsilon_{21}+\epsilon_{30} & \epsilon_{21}+\epsilon_{22} & 
    \epsilon_{22}+\epsilon_{23} &  \epsilon_{23}+\epsilon_{24} &
     \epsilon_{24}+\epsilon_{25} & \epsilon_{25}+\epsilon_{26} 
\end{array} \right]
$$
and
$$ \epsilon_X = 
\left[ \begin{array}{cccccc}
\epsilon_{11} & \epsilon_{1} &
  \epsilon_{1}+\epsilon_{2} &  \epsilon_{2}+\epsilon_{3} & 
  \epsilon_{3}+\epsilon_{4} & \epsilon_{4}+\epsilon_{6} 
 \end{array} \right]
 $$

In Table~\ref{PYTable} we list the nonsimplicities in X through these
seven perturbed wedges.  In the first perturbation, when we
perturb all nine images of $h_{12}^T$, the nonsimplicity is concentrated at the nonsimple
vertex $x_0$.  Thereafter the construction proceeds at each step by raising the dimension
of the nonsimplicity and decreasing the excess. 

The nonsimplicity does not get partitioned in this sequence of perturbed wedges, as it did
in $Y$, although we do perturb overlapping pairs of facets.  The difference is that when
we introduced the first pair of facets to be perturbed in $Y$, neither had previously been 
perturbed, so both facets shifted to the boundary in this move.  Here in $X$, the first pair
we use is $h_1^T$ and $h_2^T$; but we have already perturbed $h_1^T$, so this facet is
no longer supporting the nonsimplicity.  There is no harm in perturbing $h_1^T$ at this
point, but it has no effect on the nonsimplicity or its excess.

\begin{table}
\begin{tabular}{cccccc} \hline
$P$ & ${\rm dim}P$ & $S$ & ${\rm dim}S$ & \# supporting facets & excess \\
 & $d$ & & $k$ & $m$ & $m-(d-k)$ \\ \hline
    $\PWdg^8 P_5$ & $13$ & $x$ & $0$ & $20$ & $7$ \\ \hline
    $\PWdg \PWdg^8 P_5$ & $14$ & $x_0$ & $0$ & $20$ & $6$ \\ \hline
    $\PWdg^2 \PWdg^8 P_5$ & $15$ & $[x_0,x_t]$ & $1$ & $19$ & $5$ \\
    $\PWdg^3 \PWdg^8 P_5$ & $16$ & $I^2$ & $2$ & $18$ & $4$ \\ 
    $\PWdg^4 \PWdg^8 P_5$ & $17$ & $I^3$ & $3$ & $17$ & $3$ \\ 
    $\PWdg^5 \PWdg^8 P_5$ & $18$ & $I^4$ & $4$ & $16$ & $2$ \\ 
    $\PWdg^6 \PWdg^8 P_5$ & $19$ & $I^5$ & $5$ & $15$ & $1$ \\ \hline
    $\PWdg^7 \PWdg^8 P_5$ & $20$ & $(I^6)$ & $6$ & $14$ & ${\bf 0}$ \\ \hline
\end{tabular}
\caption{\label{PYTable}
We tabulate the nonsimplicities through the seven applications of the perturbed wedge
with $h_{33}^T$ as foot and perturbations in $X$ as prescribed by Santos and Weibel.}
\end{table}

\section{Alternate constructions}
Informed by our analysis of nonsimplicities above,
we propose a couple of alternate constructions by which to proceed from $P_5$
to a counterexample to the Hirsch conjecture via repeated applications of the 
perturbed wedge.

We need to apply eight wedges with feet in $X$ and seven wedges with feet in $Y$.
Suppose that we use eight of the original facets in $X$ for these feet, and seven of the
original facets in $Y$.  This is a different strategy than Santos and Weibel's reuse of the base
of the previous wedge as the new foot (the role of $h_{12}^T$ and $h_{33}^T$).

Let us reorder the indices of the hyperplanes in $X$ to reflect the order in which we
use them as the foot of the wedge.  Let us also reorder the indices in $Y$.  We can
then write $H^T(P_5)$ as
$$ H^T(P_5) =
\left[ \begin{array}{c} H_8^T(X) \\ H_4^T(X) \\ \hline 
H_7^T(Y) \\ H_6^T(Y) \end{array} \right].
$$

After applying the wedges at both ends, we have the blocks
$$ H^T(\Wdg^7 \Wdg^8 P_5) =
\left[ \begin{array}{ccc} H_8^T(X) & C\cdot I_8 & \cons{0} \\
H_8^T(X) & -C\cdot I_8 & \cons{0} \\ H_4^T(X) & \cons{0} & \cons{0} \\ \hline 
H_7^T(Y) & \cons{0} & C \cdot I_7\\
H_7^T(Y) & \cons{0} & -C \cdot I_7\\ 
H_6^T(Y) & \cons{0} & \cons{0} \end{array} \right].
$$

To complete the constructions, we have to apply perturbations in $X$ in the last
seven columns, and perturbations in $Y$ in the middle eight columns.  We can
accommodate the perturbations in $X$ in the upper right $8 \times 7$ block.
For the perturbations in $Y$, we will have to use two blocks as indicated below.
\begin{equation}\label{Hw15PEq}
 H^T(\PWdg^7 \PWdg^8 P_5) =
\left[ \begin{array}{ccc} H_8^T(X) & C\cdot I_8 & \epsilon(X)_7 \\
H_8^T(X) & -C\cdot I_8 & \cons{0} \\ H_4^T(X) & \cons{0} & \cons{0} \\ \hline 
H_7^T(Y) & \epsilon(Y)_7 & C \cdot I_7\\
H_7^T(Y) & \cons{0} & -C \cdot I_7\\ 
H_6^T(Y) & \epsilon(Y)_6 & \cons{0} \end{array} \right].
\end{equation}
In the following we give two different prescriptions for these perturbations.

\subsection{Singly perturbed wedge.}
In perhaps the simplest approach, we perturb a single facet in each 
application of the perturbed wedge.
The block $\epsilon(X)_7$ is an $8 \times 7$ block with the seven 
diagonal entries $\epsilon$ and otherwise $0$.
For the perturbations in $Y$, we use the $7 \times 8$ block $\epsilon(Y)_7$ with
$\epsilon$ in the seven super-diagonal entries and otherwise $0$; and for the
eighth perturbation we use $\epsilon(Y)_6$, setting the upper-left entry to $\epsilon$
and the rest $0$.

\subsection{Pinched perturbed wedge.}
This is perhaps the closest in spirit to Santos' desire to produce all-but-simple spindles.
In this approach, after each wedge we perturb two facets, one with a positive perturbation
and one negative.  This pinches the non singularity off at a vertex again.  The nonsingularity
tries to grow in dimension, for example to the edge $[y_b,y_t]$.  The positive perturbation
truncates away the segment $[y_b,y_0)$; now the negative perturbation truncates away
the segment $(y_0,y_t]$, leaving an edge of length $0$ -- the nonsimple vertex $y_0$.

For the matrix for this construction we again use Equation~\ref{Hw15PEq}.
Here the block $\epsilon(X)_7$ is an $8 \times 7$ block with the seven 
diagonal entries $\epsilon$, the seven sub-diagonal entries $-\epsilon$, 
and otherwise $0$.

For the perturbations in $Y$, we need to use two matrix blocks to cover the
eight perturbations. We mimic the bands in $\epsilon(X)_7$ by using the two blocks
$\epsilon(Y)_7$ and $\epsilon(Y)_6$.  In the $7 \times 8$ block $\epsilon(Y)_7$,
we put $\epsilon$ in the seven super-diagonal entries, $-\epsilon$ in the six entries of
the next diagonal, and otherwise $0$.  We pick up the remnant again in $\epsilon(Y)_6$
by making the lower left entry $\epsilon$, the two entries on the adjacent diagonal
$-\epsilon$, and otherwise $0$.

This nonsimple vertex has the combinatorial structure of a $k$-dimensional hypercube
with boundary of zero length.

\subsection{General requirements on the $\epsilon$-blocks.}
What do we require of the perturbations in $X$ and $Y$ in order to produce
a simple polytope?  Let's look at the general system
\[ H(\PWdg^7 \PWdg^8 P_5) = \left[ \begin{array}{cc|c}
\begin{array}{c} H_8^T(X)  \\ H_8^T(X)  \\
H_4^T(X) \end{array} & 
\begin{array}{c}  C \cdot I_{8\times 8} \\  -C \cdot I_{8\times 8} \\ \cons{0} \end{array} & 
\epsilon(X) \\ \hline
\begin{array}{c} H_7^T(Y)  \\ H_7^T(Y) \\ H_6^T(Y) \end{array} & 
\epsilon(Y) &
\begin{array}{c} C \cdot I_{7\times 7} \\  -C \cdot I_{7\times 7} \\ \cons{0} \end{array} 
\end{array}
\right] \]
This system encodes the inward-pointing normals, 
$H^T V \ge \cons{0}$, of the perturbed wedges when we 
don't reuse facets for the feet of the wedges. 

Let's consider the perturbations in $X$, encoded in the $20 \times 7$ block
$\epsilon(X)$.  Before the perturbations, when $\epsilon(X)=\cons{0}$, each
image of $x$ can be indexed by $b$ if it is incident to the base (corresponding to
entries in $C\cdot I$) or $t$ if incident to the top (corresponding to entries
in $-C \cdot I$) for each of these dimensions.
\[ x_{b\cdots b} = \left[\begin{array}{c} x \\ \hline \cons{0} \\ \hline \cons{-}
\end{array} \right], \; \ldots , \; 
x_{t \cdots t} = \left[\begin{array}{c} x \\ \hline \cons{0} \\ \hline \cons{+}
\end{array} \right]
\]
At this point, each image of $x$ is incident to all $20$ of the facets in $X$
and half of the $14$ facets that are bases or tops of the wedges in $Y$.

Before perturbations in $X$, the nonsimplicity is a $7$-dimensional
hypercube supported by the $20$ hyperplanes in $X$ in a $20$-dimensional 
polytope.  So the excess is $20-(20-7)=7$.  We therefore need the perturbations
in $\epsilon(X)$ to have ${\rm rank} (\epsilon(X)) = 7$.

Similarly we need ${\rm rank} (\epsilon(Y)) = 8$.

Any positive perturbations truncate away the base vertices in that dimension, and 
any negative perturbations truncate away the top vertices in that dimension.  If in
a given dimension we include both positive and negative perturbations, the vertical
segments in that dimension are pinched off at coordinate $0$. 

We note finally that there are perturbations in this general setting that do not 
correspond to sequential applications of the perturbed wedge.  As a simple example of
this, if in the first column of $\epsilon(X)$ we perturb only the first facet in $X$, and in the first column of $\epsilon(Y)$ we perturb only the first facet in $Y$.  Under a sequential application
of perturbed wedges, we would see one or the other of these duplicated.

\section{Summary}
Santos' construction of the first known counterexample to the Hirsch conjecture, for bounded
polytopes, follows the strategy of first finding a counterexample to the nonrevisiting conjecture.
In this case Santos constructs a $5$-dimensional all-but-simple spindle $(P,x,y)$ of length $6$.

For simple polytopes, if we had a counterexample to the nonrevisiting conjecture, we would
produce the corresponding counterexample to the Hirsch conjecture through repeated wedging,
over all the facets not incident to $x$ or $y$.  However, Santos' $5$-dimensional spindle is not simple.  Every facet is incident to either $x$ or $y$,
so we need an alternate method to produce the corresponding counterexample to the Hirsch
conjecture.  Santos has offered the perturbed wedge to accomplish this.

In the counterexample posted in 2012, Santos and Weibel begin with an all-but-simple spindle
in dimension $5$, with $25$ facets and length $6$.  They use this as the starting point
for applying the perturbed wedge $15$ times, to obtain a counterexample to the Hirsch
conjecture in dimension $20$.

We have used Santos and Weibel's specific counterexample to guide our exploration of several technical
details underlying the perturbed wedge construction.  We have shown that there are constraints
on the selection of a facet to be the foot of the wedge and also on the selection of the facet to be
perturbed.  We have also seen how the nonsimplicity evolves under different strategies for the
perturbations.

\section*{Appendix: Santos and Weibel's all-but-simple spindle.}
\renewcommand{\arraystretch}{1.2}

In the preceding we explore the perturbed wedge as it was applied 
in Santos and Weibel's counterexample as posted in 2012 \cite{Weibel}.  The underlying polytope
is a $5$-dimensional all-but-simple spindle $(P,x,y)$ with $25$ facets and length $6$.  
The inward-pointing normals to the facets are given by:
\begin{eqnarray*}
H^T (P)  =  \left[ \begin{array}{c} X \\ Y \end{array} \right] 
 &=& \scriptstyle \left[ \begin{array}{rrrrrr}
 1 & 100 & 0 & 0 & 21 & -7 \\
 1 & 100 & 0 & 0 & 16 & -15 \\
 1 & 100 & 0 & 0 & 0 & -32 \\
 1 & 100 & 0 & 0 & -16 & -15 \\
 1 & 100 & 0 & 0 & -21 & -7 \\
 1 & 100 & 0 & 0 & -20 & -4 \\
 1 & 100 & 0 & 0 & 0 & 32 \\
 1 & 100 & 0 & 0 & 20 & -4 \\
 1 & 100 & \scriptstyle 3/100 & \scriptstyle -1/50 & 0 & -30 \\
 1 & 100 & \scriptstyle -3/100 & \scriptstyle -1/50 & 0 & 30 \\
 1 & 100 & \scriptstyle -3/2000 & \scriptstyle 7/2000 & 0 & \scriptstyle 318/10 \\
 h_{12}^T= 1 & 100 & \scriptstyle 3/2000 & \scriptstyle 7/2000 & 0 & \scriptstyle -318/10 \\  \hline
 h_{21}^T = 1 & -100 & 30 & 0 & 0 & 0 \\
 1 & -100 & 4 & -15 & 0 & 0 \\
 1 & -100 & 0 & \scriptstyle -33/2 & 0 & 0 \\
 1 & -100 & -1 & -16 & 0 & 0 \\
 1 & -100 & \scriptstyle -55/2 & 0 & 0 & 0 \\
 1 & -100 & -17 & 18 & 0 & 0 \\
 1 & -100 & 0 & 38 & 0 & 0 \\
 1 & -100 & 22 & 17 & 0 & 0 \\
 1 & -100 & -10 & 0 & \scriptstyle 1/5 & \scriptstyle -1/5 \\
 1 & -100 & \scriptstyle 2999/100 & 0 & \scriptstyle -3/25 & \scriptstyle -1/5 \\
 1 & -100 & \scriptstyle 299999/10000 & 0 & 0 & \scriptstyle 1/100 \\
 1 & -100 & \scriptstyle -2745/100 & 0 & \scriptstyle 1/5000 & \scriptstyle 1/800 \\
h_{33}^T =    1 & -100 & -27 & 0 & \scriptstyle 1/500 & \scriptstyle -1/80 \\ 
 \end{array} \right]
\end{eqnarray*}
We index the rows to be consistent with Santos and Weibel's completed counterexample to the Hirsch conjecture.
The indices $h_{13}^T,\ldots,h_{20}^T$ are reserved for the applications of the wedge
at the vertex $x$, and the indices $h_{34}^T, \ldots,h_{40}^T$ are reserved for the applications
of the wedge at the vertex $y$.

\bibliographystyle{amsplain}
\providecommand{\bysame}{\leavevmode\hbox to3em{\hrulefill}\thinspace}
\providecommand{\MR}{\relax\ifhmode\unskip\space\fi MR }
\providecommand{\MRhref}[2]{%
  \href{http://www.ams.org/mathscinet-getitem?mr=#1}{#2}
}
\providecommand{\href}[2]{#2}

\end{document}